\documentclass[a4paper,11pt,english]{amsart}

\usepackage{multirow}
\usepackage[english]{babel}
\usepackage[latin1]{inputenc}
\usepackage[T1]{fontenc}

\usepackage{graphicx}
\usepackage{hyperref}
\usepackage{amssymb}
\usepackage{amsmath}
\usepackage{verbatim,enumerate}
\usepackage{tikz}

\newtheorem{thm}{Theorem}
\newtheorem{lemma}[thm]{Lemma}
\newtheorem{prop}[thm]{Proposition}

\newtheorem{defi}[thm]{Definition}
{\theoremstyle{definition}
\newtheorem{exa}[thm]{Example}
\newtheorem{rem}[thm]{Remark}}

\newcommand{\C}{\mathbb{C}}

\renewcommand{\epsilon}{\varepsilon}
\newcommand{\R}{\mathbb{R}}
\newcommand{\Z}{\mathbb{Z}}

\newcommand{\CC}{\widetilde C}

\usepackage{wasysym}

\begin{document}

\title{Pseudoholomorphic simple Harnack curves}

\author{Erwan Brugall\'e}
\address{\'Ecole Polytechnique,
Centre Math\'ematiques Laurent Schwartz, 91 128 Palaiseau Cedex, France}

\email{erwan.brugalle@math.cnrs.fr}

\subjclass[2010]{Primary 14P25, 32Q65; Secondary 14P05}
\keywords{Pseudoholomorphic curves, Harnack curves}

\begin{abstract}
We give a new proof of Mikhalkin's Theorem on the
 topological classification  of 
simple Harnack curves, which in particular extends
 Mikhalkin's result to real pseudoholomorphic curves.
\end{abstract}
\maketitle

A non-singular (abstract) 
\emph{real algebraic curve} is a non-singular complex algebraic
curve $C$ equipped with an anti-holomorphic involution $conj_C$. The real 
part of $C$, denoted by $\R C$, is by definition the set of fixed
points of $conj_C$. If $C$ is compact, then $\R C $ is a disjoint
union of at most
$g(C)+1$ smooth circles, where $g(C)$ is the genus of $C$.
When $\R C$ has precisely $g(C)+1$ connected components, we say that the real
curve $C$ is \emph{maximal}. Equivalently, a real algebraic curve
 $C$ is maximal if and only the quotient $C/conj_C$ is a disk with
$g(C)$ holes
(see for example \cite{V3}).

A \emph{real  map} $\phi:C \to \C P^2$ from a real algebraic curve
 is a map such that $\phi\circ
conj_C=conj\circ \phi$, where $conj([x:y:z])=[\overline x:\overline
  y:\overline z]$ is the standard complex conjugation on $\C
P^2$. Note that $\phi(\R C)\subset \R\phi(C)$ if $\phi$ is real, however
this inclusion might be strict as $\phi$ may map pairs of
$conj_C$-conjugated points 
 to $\R P^2$.
Given $\phi:C \to \C P^2$ a real smooth map, 
a point $p\in \R \phi(C)$ is
called a \emph{solitary node} if there exists a neighborhood $U$ of
$p$ in $\R P^2$ such that  $\phi^{-1}(U)=\phi^{-1}(p)$ which in
addition consists of two 
$conj_C$-conjugated points
 at which the differential of $\phi$ is
injective (i.e. locally at $p$, $\phi(C)$ is the transverse
intersection of two complex conjugated disks).

\section{Introduction}\label{sec:intro}

Let $L_0,L_1,$ and $L_2$ be three distinct real lines in $\C P^2$ with
$L_0\cap L_1\cap L_2=\emptyset$.
A \emph{simple Harnack curve} is  a real algebraic
map
$\phi:C\to\C P^2$ 
satisfying the two following conditions:
\begin{itemize}
\item $C$ is a non-singular maximal real algebraic curve;
\item there exist a connected component 
 $\mathcal O$ of $\R C$, and three disjoint
  arcs $l_0,l_1,l_2$ contained in $\mathcal O$ such that $\phi^{-1}(L_i)\subset l_i$.
\end{itemize}

Note that by Bézout Theorem, the set $\phi^{-1}(L_i)$ contains
finitely many points.
We depict in Figure~\ref{fig:harnack} examples of simple Harnack
curves with a non-singular image in $\C P^2$ and 
intersecting transversely 
all lines $L_i$. 
Theorem~\ref{thm:harnack} below says that these are essentially the only
simple Harnack curves.
\begin{figure}[h!]
\centering
\begin{tabular}{ccc} 
\includegraphics[height=5cm, angle=0]{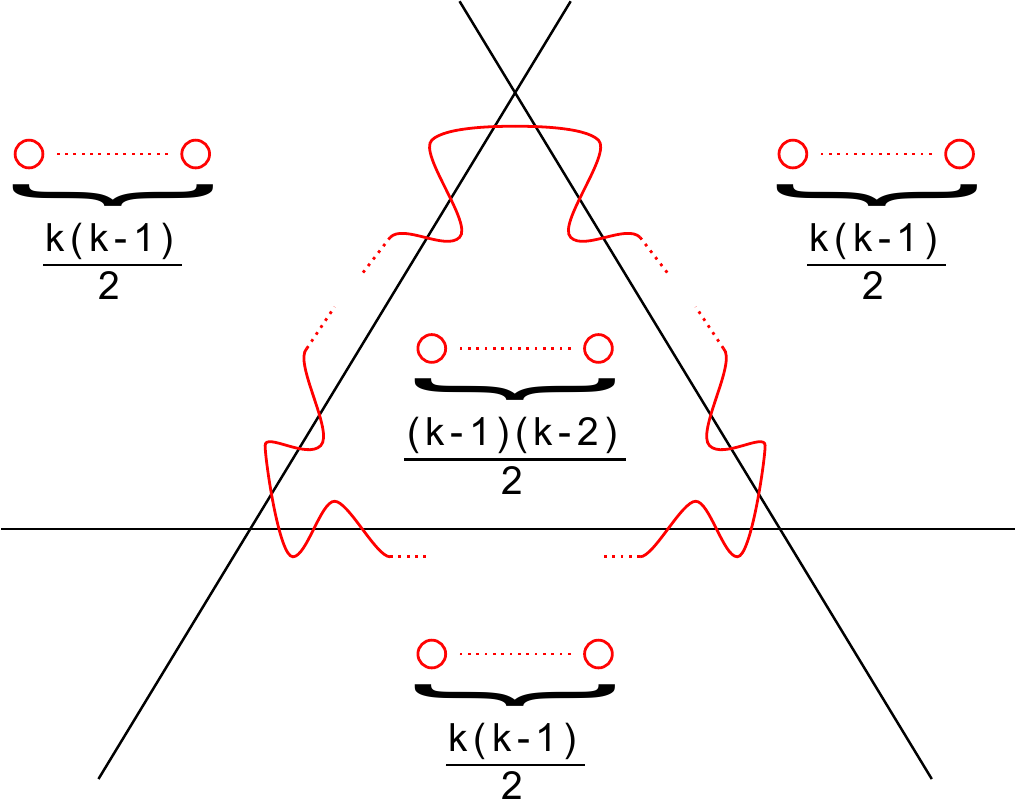}& \hspace{5ex} &
\includegraphics[height=5cm, angle=0]{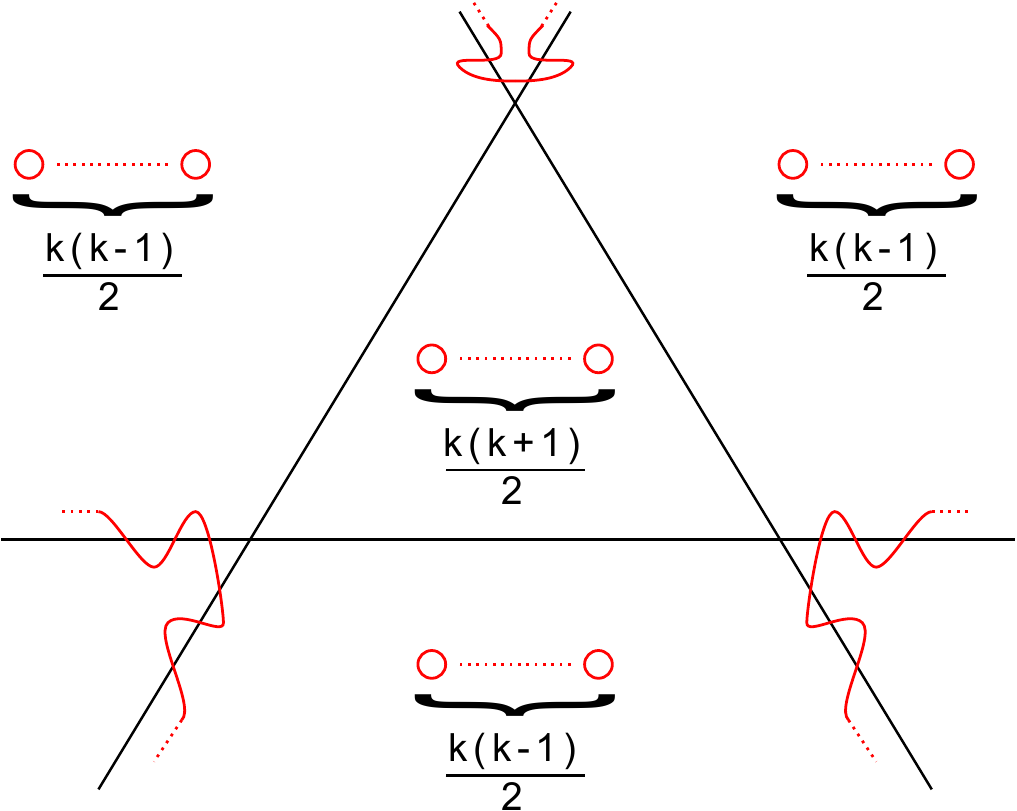}
\\ \\ a) $d=2k$ && b) $d=2k+1$
\end{tabular}
\caption{Simple Harnack curves of degree $d$ and genus
  $\frac{(d-1)(d-2)}{2}$; 
in particular three quadrants of $\R
  P^2\setminus (\bigcup_{i=0}^2)\R L_i$ contain $\frac{k(k-1)}{2}$
  circles in $\phi(\R C)$, 
while the fourth one contains either $\frac{(k-1)(k-2)}{2}$ or 
$\frac{k(k+1)}{2}$ such circles depending on the parity of $d$.}
\label{fig:harnack}
\end{figure}

Let $\phi:C\to \C P^2$ be a simple Harnack
curve,
and choose an orientation of $\mathcal O$. This induces an ordering of 
the intersection points  of $\mathcal O$ (or $C$) with  $L_i$, and 
we
denote by $s_i$ the corresponding sequence of
intersection multiplicities. Let $s$ be the sequence $(s_0,s_1,s_2)$
considered up to the equivalence relation generated by
$$
(s_0,s_1,s_2)\sim (\overline s_0,\overline s_1,\overline s_2),
\quad(s_0,s_1,s_2)\sim (s_2,s_0,s_1),\quad\mbox{and}\quad 
(s_0,s_1,s_2)\sim 
( s_0,s_2, s_1),$$
where $\overline{(u_i)_{1\le i\le n}}=(u_{n-i})_{1\le i\le n}$.
This equivalence relation is so that 
  $s$ does not depend on the chosen orientation on
$\mathcal O$, nor on the labeling of the three lines $L_i$.

\begin{thm}[Mikhalkin \cite{Mik11}, Mikhalkin-Rullg{\aa}rd
    \cite{MikRul01}]\label{thm:harnack} 
Let $\phi:C\to\C P^2$ be a simple Harnack curve of degree $d$, and
suppose that $\phi(C)$ is the limit of images of  a sequence of 
 simple Harnack curves of degree $d$ and genus  $g(C)=\frac{(d-1)(d-2)}{2}$.
 Then 
the curve $\phi(C)$ 
has  solitary nodes 
 as only singularities (if any).
Moreover if either $g(C)=0$ or $g(C)=\frac{(d-1)(d-2)}{2}$, then
the topological type of the
pair $\left(\R P^2, \R \phi(C)\bigcup_{i=0}^2\R L_i\right)$ only
depends on $d$, 
$g(C)$, and $s$.
\end{thm}
Mikhalkin actually proved Theorem~\ref{thm:harnack} 
 for simple Harnack curves in any toric surface, nevertheless this
 a priory more general statement can  be deduced from the particular 
case of $\C
 P^2$, see Appendix~\ref{app:other toric}.
Existence of simple Harnack curves of maximal genus with any Newton polygon, and
intersecting transversely all toric divisors, was first established by
Itenberg 
(see \cite{IV2}). Simple Harnack curves of any degree, genus, and
sequence $s$  
 were first
constructed by Kenyon and
Okounkov in \cite{KenOko06}. In addition, when $g=0$ they could
remove the hypothesis that $\phi(C)$ has to be
 the limit of images of  a sequence of 
 simple Harnack curves of degree $d$ and genus  $g(C)=\frac{(d-1)(d-2)}{2}$.
In Theorem~\ref{thm:pseudoharnack} below, we remove this hypothesis
for any $g$.

Because they are extremal objects, simple Harnack curves play an
important role
 in real
algebraic geometry, and 
Theorem~\ref{thm:harnack} had a deep impact on subsequent developments in
this field. However their importance goes beyond  real geometry,
as showed their connection to dimers discovered by Kenyon,
Okounkov, and Sheffield in  \cite{KeOkSh06}.

\bigskip
The goal of this note is to give an alternative proof of Theorem~\ref{thm:harnack}. 
Moreover, our
proof is also valid for \emph{real pseudoholomorphic curves},
which 
are also very important objects in real algebraic and symplectic geometry. Note that
a real algebraic map $\phi:C\to \C P^2$ is  pseudoholomorphic, however
the converse is not true in general.
Mikhalkin's original proof of Theorem~\ref{thm:harnack}
uses amoebas of algebraic curves, 
and does not a priory apply to real pseudoholomorphic maps which
are not algebraic.

It is nevertheless possible to read our proof of 
Theorem~\ref{thm:harnack} in the algebraic category, by  going
directly to Section~\ref{sec:algebraic}, and defining the map $\pi_i: C\to
L_i$ as the composition of  $\phi$ with the linear projection 
$\C P^2\setminus\left(L_j\cap L_k\right)\to L_i$, with
$\{i,j,k\}=\{0,1,2\}$.

\medskip
We consider $\C P^2$ equipped with the standard Fubini-Study symplectic
form $\omega_{FS}$.
Recall that an almost complex structure
 $J$ on $\C P^2$ 
is said to be \emph{tamed} by $\omega_{FS}$ if
 $\omega_{FS}(v,Jv)>0$ for any
non-null vector $v\in  T\C P^2$.
 Such an almost complex structure
is called \emph{real} if 
the standard complex conjugation $conj$ on
$\C P^2$ is $J$-antiholomorphic  (i.e. $conj\circ J=J^{-1}\circ conj$).
For example, the standard complex structure 
on $\C P^2$ is a real almost complex structure.

Let $(C,\omega)$ be a compact symplectic surface equipped with a 
complex structure $J_C$ tamed by $\omega$, and  a
$J_C$-antiholomorphic involution $conj_C$, and
let
 $J$ be a real  almost complex structure on $\C P^2$.
A symplectomorphism $\phi:C\to \C P^2$ is a 
\emph{real   $J$-holomorphic map} if
$$d\phi\circ J_C=J\circ d\phi \quad\mbox{and}\quad \phi\circ
  conj_C=conj\circ\phi. $$
It is of degree $d$ if
$\phi_*([C])=d[\C P^1]$ in $H_2(\C P^2;\Z)$.
Recall that 
any intersection of two $J$-holomorphic curves is positive (see
  {\cite[Appendix E]{McS}}). 

The definition of simple Harnack
  curves extends immediately to the real $J$-holomorphic case.
Given three distinct real
$J$-holomorphic lines $L_0$, $L_1$, and $L_2$ in $\C P^2$ such that
$\bigcap_{i=0}^2 L_i=\emptyset$,
a real $J$-holomorphic curve $\phi:C \to\C P^2$ 
is a simple Harnack curve if $C$ is maximal, and if 
there exists a connected component 
 $\mathcal O$ of $\R C$, and three disjoint
  arcs $l_0,l_1,l_2$ contained in $\mathcal O$ such that
  $\phi^{-1}(L_i) \subset l_i$.  

\begin{thm}\label{thm:pseudoharnack}
Let $\phi:C\to\C P^2$ be a $J$-holomorphic 
simple Harnack curve of degree $d$.
 Then 
the curve $\phi(C)$ 
has  solitary nodes 
 as only singularities (if any).
Moreover
if either $g(C)=0$ or $g(C)=\frac{(d-1)(d-2)}{2}$, then
 the topological type of the pair 
$\left(\R P^2, \R \phi(C)\bigcup_{i=0}^2\R L_i\right)$ does not depend
 on $J$, once $d$ and 
$s$ are fixed.
\end{thm}
It follows from  Theorem~\ref{thm:pseudoharnack} that
Figure~\ref{fig:harnack}
is enough to recover all topological types of pairs 
$\left(\R
P^2, \R \phi(C)\bigcup \cup_{i=0}^2\R L_i\right)$ where $\phi:C\to \C
P^2$ is
 a simple Harnack
curve, see Appendix~\ref{app:all harnack}.
As in the case of algebraic curves, 
one may generalize Theorem~\ref{thm:pseudoharnack} to   $J$-holomorphic
simple Harnack curves in any toric surface, see Appendix~\ref{app:other toric}.

The proof of Theorem~\ref{thm:pseudoharnack} goes along the following lines:
the three projections from $\C P^2\setminus\left(L_j\cap L_k\right)$
to $L_i$
induce three 
ramified coverings
$\pi_i:C\to L_i$;
by considering the arrangement of the real Dessins d'enfants 
$\pi_i^{-1}(\R L_i)$ on $C/conj_C$, we deduce the number of connected
components of $ \R \phi(C)$ in 
each quadrant 
of $\R P^2\setminus \left(\cup_{i=0}^2\R L_i\right)$, as well as its
complex orientation; the mutual position of all these connected
components is then deduced from Rokhlin's complex orientation formula.

\medskip
{\bf Acknowledgment: }I am grateful to Ilia Itenberg and Grigory
Mikhalkin for their comments on a preliminary exposition of this work.
I also thank Patrick Popescu-Pampu,
Jean-Jacques Risler, and Jean-Yves Welschinger for
 their valuable remarks on  a first version of this note.

\section{Proof of Theorem \ref{thm:pseudoharnack}}

Let $\phi:C\to\C P^2$ be a $J$-holomorphic simple Harnack curve in $\C P^2$
of degree $d$ and genus $g$.
We define $p_{i,j}=L_i\cap L_j$.

\subsection{Construction of the maps $\pi_i:C\to L_i$}
Gromov proved in  \cite{Gro} that there exists a unique
$J$-holomorphic line passing  through
 two distinct points in $\C P^2$.
By uniqueness, this
line is real if the two points are in $\R P^2$, hence there exists a
real pencil of
$J$-holomorphic lines  through any point of $\R P^2$.
In particular if 
$\{i,j,k\}=\{0,1,2\}$, the  map $\C P^2\setminus \{p_{j,k}\}\to L_i$, which associates to
each point $p$  the unique intersection point of $L_i$
with the $J$-holomorphic  line passing through $p$ and $p_{j,k}$,
 is a real 
smooth map. We define $\pi_i:C\to L_i$ as the
composition of $\phi$ with this projection. 
By positivity of intersections of $J$-holomorphic curves, 
the map $\pi_i$ is a real ramified covering.

\subsection{Dessins d'Enfants on $C$}\label{sec:algebraic}

We denote by $\widetilde C$  the quotient of $C$ by $conj_C$. Since $C$ is maximal, the surface $\CC$ is a disk with $g$
holes. 

Let $\Gamma_i\subset \CC$ be the graph $\pi_i^{-1}(\R L_i)/conj_C$. 
Note that $\Gamma_j\cap \Gamma_k=\phi^{-1}(\R P^2)$ if $j\ne k$,
in particular $\Gamma_j\cap \Gamma_k=\bigcap_{i=0}^2\Gamma_i$.
We call a \emph{triple point} 
an isolated point in $\bigcap_{i=0}^2\Gamma_i$. By construction, a triple
point  corresponds to a singular point of $\phi(C)$ in $\R P^2$, where at
least two complex conjugated non-real branches intersect. 
By the adjunction formula (see {\cite[Chapter 2]{McS}} in the
case of $J$-holomorphic curves), the graph $\bigcup_{i=0}^2\Gamma_i$
has no more than $\frac{(d-1)(d-2)}{2} -g$ triple points, and
$\phi(C)$ is nodal with only solitary nodes in case of equality.

Let $\{i,j,k\}=\{0,1,2\}$. We label by $+$ (resp. $-$) the connected
component of $\R L_i\setminus\{p_{i,j},p_{i,k}\}$ containing
(resp. disjoint from)
$\phi(\mathcal O)\cap L_i$. 
We endow each connected component of
$\Gamma_i\setminus\pi_i^{-1}\left(\{p_{i,j},p_{i,k}\}\right)$ with the
sign of the corresponding component of $\R L_i\setminus\{p_{i,j},p_{i,k}\}$.
We also denote by
$(\epsilon_0,\epsilon_1)\in\{+,-\}^2$ the connected component of 
$\R P^2\setminus \left(\bigcup_{i=0}^2 \R L_i\right)$ which project to
the components 
labeled by $\epsilon_0$ and $\epsilon_1$ of $\R L_0$ and $\R L_1$
under the projections of center $p_{1,2}$ and $p_{0,2}$ respectively.

\medskip
The map $\pi_i:C\to L_i$ is 
a ramified covering 
of degree $d$, so by the 
Riemann-Hurwitz formula it has exactly 
$2(d+g-1) $
ramification points (counted with multiplicity). 
Given $j\ne i$, a subarc of 
$l_j$ connecting 
two  consecutive
points in $l_j\cap \phi^{-1}(L_j)$ has to contain  a ramification
point of $\pi_i$ in its interior, and a point of contact of order $c$
of $l_j$ with $\R L_j$ is a ramification point of multiplicity $c-1$
of $\pi_i$. Alltogether, 
the set $l_j\cup l_k$ with $\{i,j,k\}=\{0,1,2\}$
contains at least $2(d-1)$ ramification points of $\pi_i$ (counted
with multiplicity).
Moreover a connected component of $\R C$ distinct from $\mathcal O$ contains
at least two ramification points of $\pi_i$. Since $C$ has $g+1$
connected components, if follows that these two previous lower bounds
are in fact equality, in particular all ramification points of $\pi_i$ are real. 
This implies that each connected component of $\CC\setminus \Gamma_i$
is a disk, and that the restriction of $\pi_i$ on this disk is a
homeomorphism to one the two hemispheres of $L_i\setminus\R L_i$.

\begin{lemma}\label{lem:g=0}
If $g=0$, then
the arrangement of $\bigcup_{i=0}^2\Gamma_i$ in $\CC$ only depends,  up to
 orientation preserving homeomorphism, on $d$ and $s$.
In particular it 
has exactly
$\frac{(d-1)(d-2)}{2}$ triple points.
\end{lemma}
\begin{proof}
Since $\pi_i$ has no ramification point outside $\mathcal O$, the
graph $\Gamma_i$ decomposes $\CC$ into a chain of disks, where two
adjacent disks intersect along (the closure of) a connected component
of $\Gamma_i\setminus \mathcal O$. See Figure~\ref{fig:disk} in the case when
$d=6$ and $\phi^{-1}(L_i)$ consists of $6$ distinct points. By
definition,
 the points of $\Gamma_i$ in $l_i$ are endowed with the sign $+$. 
\begin{figure}[h!]
\centering
\begin{tabular}{ccc} 
\includegraphics[height=5cm, angle=0]{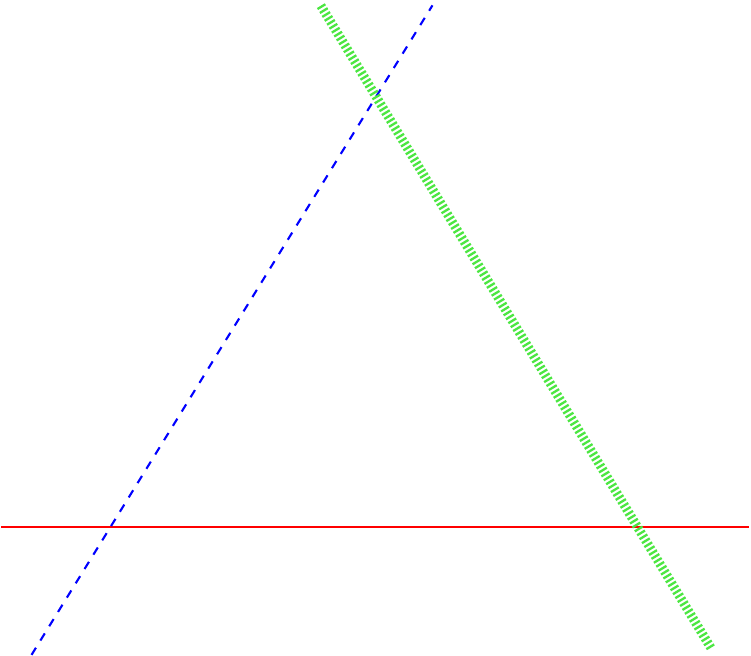}& \hspace{5ex}
\put(-100, 15){$\R L_0$}
\put(-140, 80){$\R L_1$}
\put(-60, 80){$\R L_2$}
\put(-170, 35){$p_{0,1}$}
\put(-30, 35){$p_{0,2}$}
\put(-85, 120){$p_{1,2}$}
 &
\includegraphics[height=5cm, angle=0]{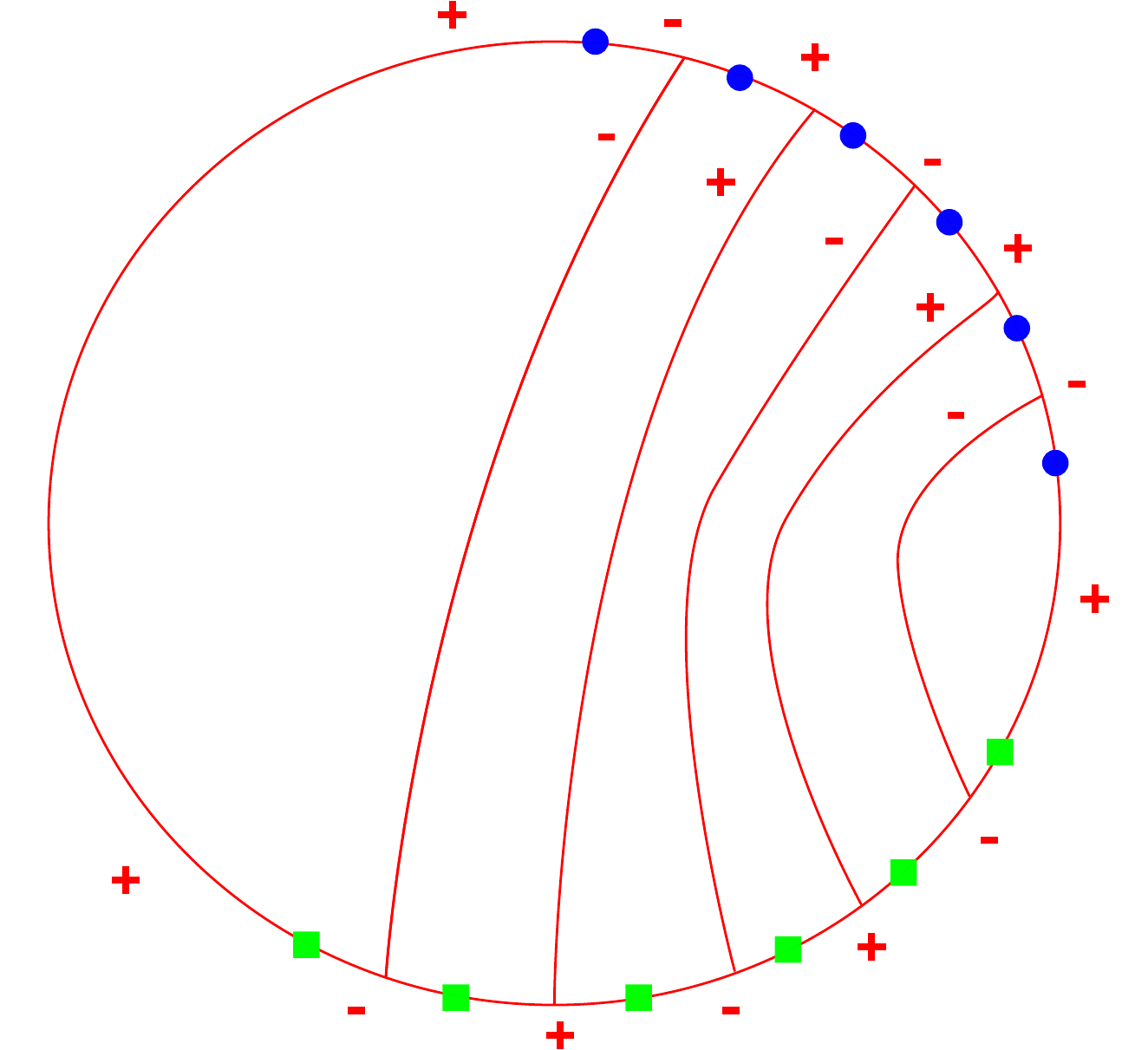}
\\ 
\\ a) $\R L_0\cup \R L_1\cup \R L_2$ in $\R P^2$& & b) The graph $\Gamma_0=\pi_0^{-1}(L_0)$, dots and squares 
\\ && $\ \ \ \  $ being points in $\phi^{-1}(L_1)$ and $\phi^{-1}(L_2)$ respectively
\end{tabular}
\caption{}
\label{fig:disk}
\end{figure}

By the adjunction formula, the number of intersection points of
the graphs $\Gamma_i$ and $\Gamma_j$, with
$i\ne j$, is not more than
$\frac{(d-1)(d-2)}{2}=1+2+\ldots+d-2$. However, this number is clearly
the minimal number of intersection point of $\Gamma_i$ and $\Gamma_j$,
and  there exists a unique mutual position of those graphs that
achieves this lower bound (see Figure~\ref{fig:disk2}a). The lemma
follows immediately by symmetry  (see Figure~\ref{fig:disk2}b).
\begin{figure}[h!]
\centering
\begin{tabular}{ccc} 
\includegraphics[height=5cm, angle=0]{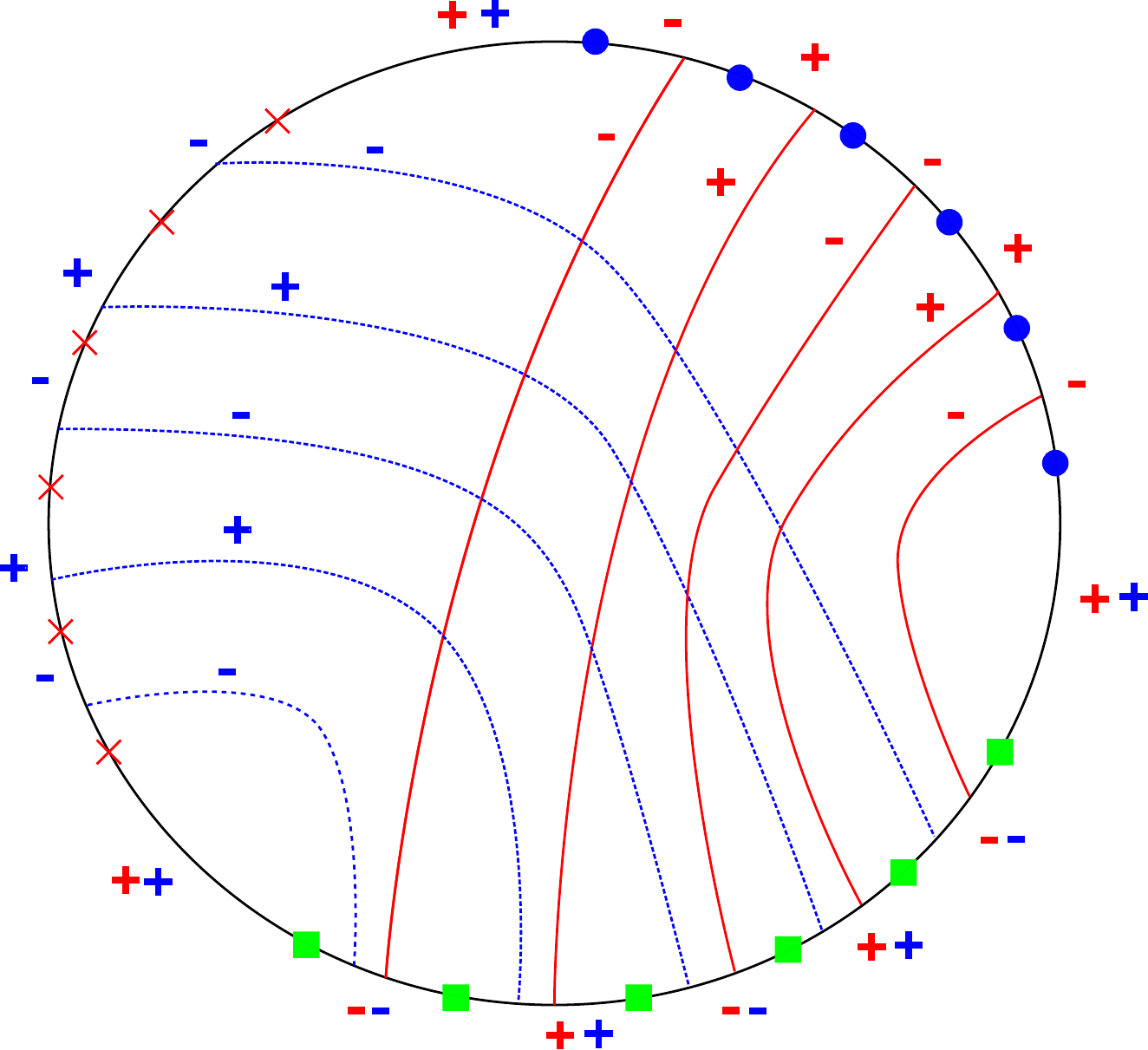}& \hspace{5ex} &
\includegraphics[height=5cm, angle=0]{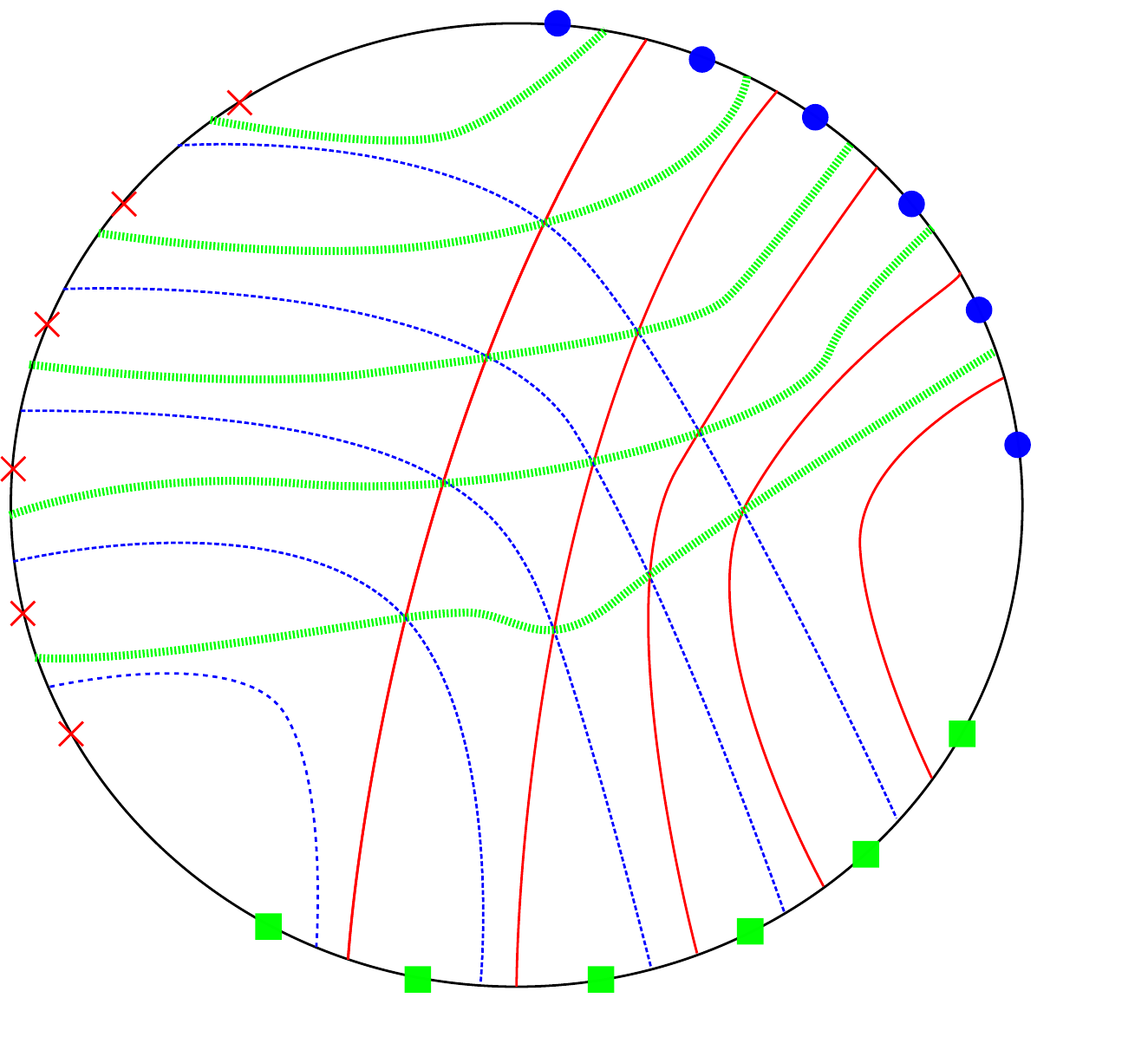}
\\ \\ a) $\Gamma_0\cup\Gamma_1$, crosses being points in $\phi^{-1}(L_0)$
&& b) $\Gamma_0\cup\Gamma_1\cup\Gamma_2$ 
\end{tabular}
\caption{}
\label{fig:disk2}
\end{figure}
\end{proof}

In case of positive genus, we have the following lemma.

\begin{lemma}\label{lem:g>0}
The
 arrangement of $\bigcup_{i=0}^2\Gamma_i$ 
has exactly
$\frac{(d-1)(d-2)}{2}-g$ triple points.
Moreover if $d=2k$ (resp. $d=2k+1$), then $\R \phi(C)$ has exactly
$\frac{(k-1)(k-2)}{2}$ 
(resp. $\frac{k(k+1)}{2}$)
connected components in the quadrant  $(+,+)$ (resp. $(-,-)$), and 
$\frac{k(k-1)}{2}$ 
connected components in each of the other quadrants.
\end{lemma}
\begin{proof}
Locally around each boundary component of $\CC$ distinct from
$\mathcal O$, the graph $\bigcup_{i=0}^2\Gamma_i$  looks like in
Figure~\ref{fig:g>0}a. In particular, we may glue a disk 
as depicted in Figure~\ref{fig:g>0}b. Performing this operation
to each boundary component of $\CC$ distinct from
$\mathcal O$, the lemma is proved with the same arguments as Lemma~\ref{lem:g=0}.
\begin{figure}[h!]
\centering
\begin{tabular}{ccc} 
\includegraphics[height=5cm, angle=0]{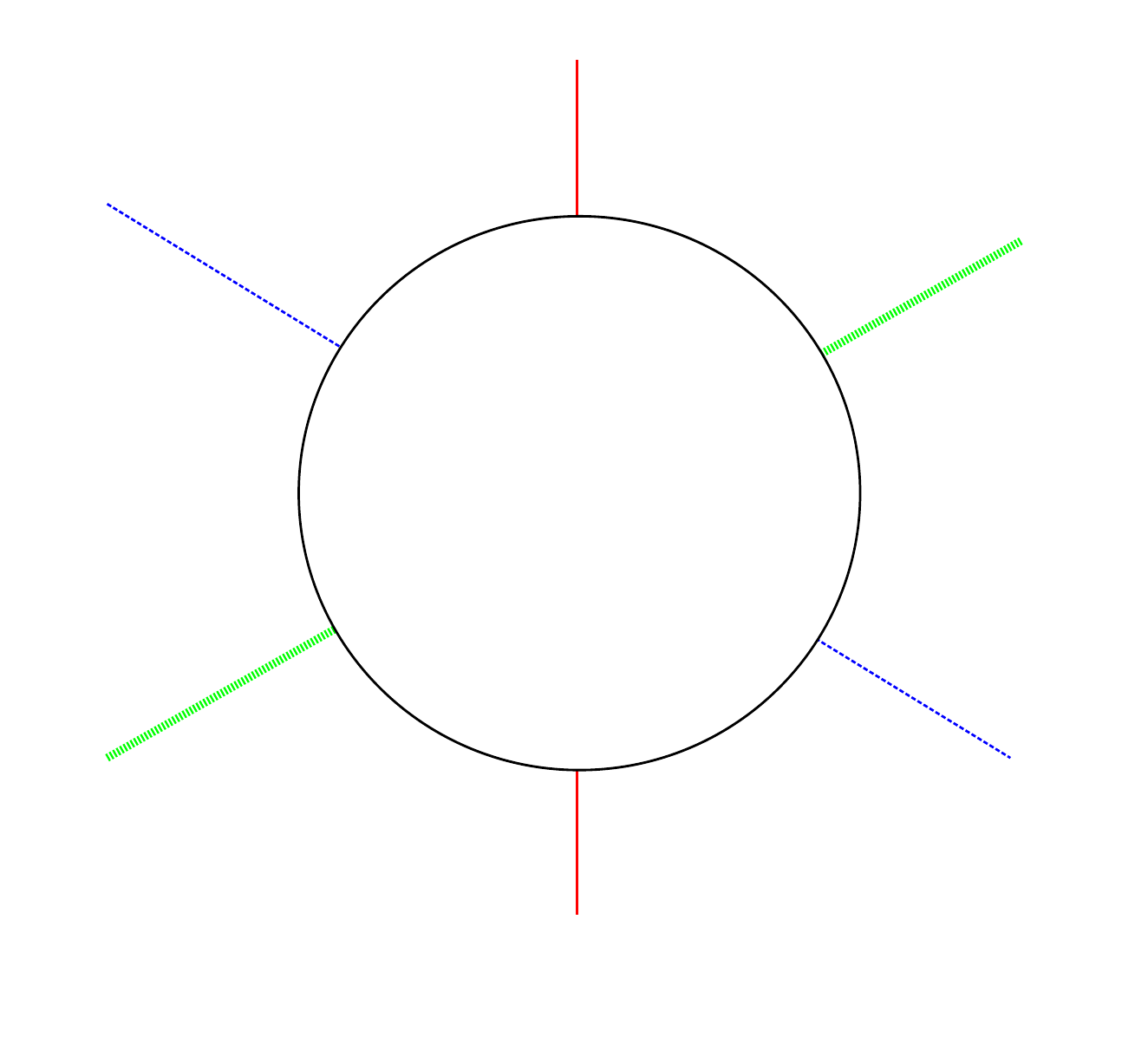} &  
\includegraphics[height=5cm, angle=0]{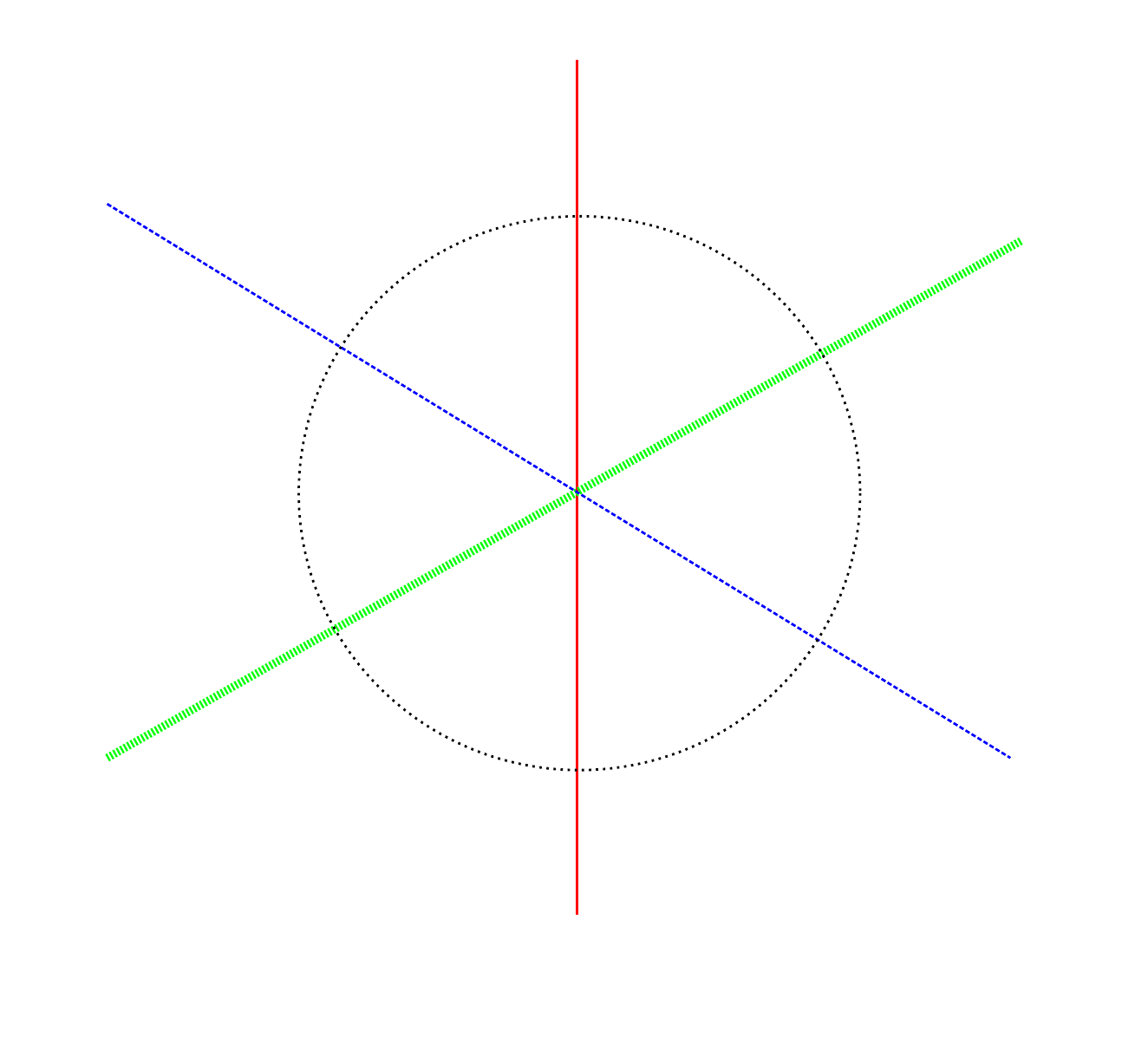} & 
\includegraphics[height=5cm, angle=0]{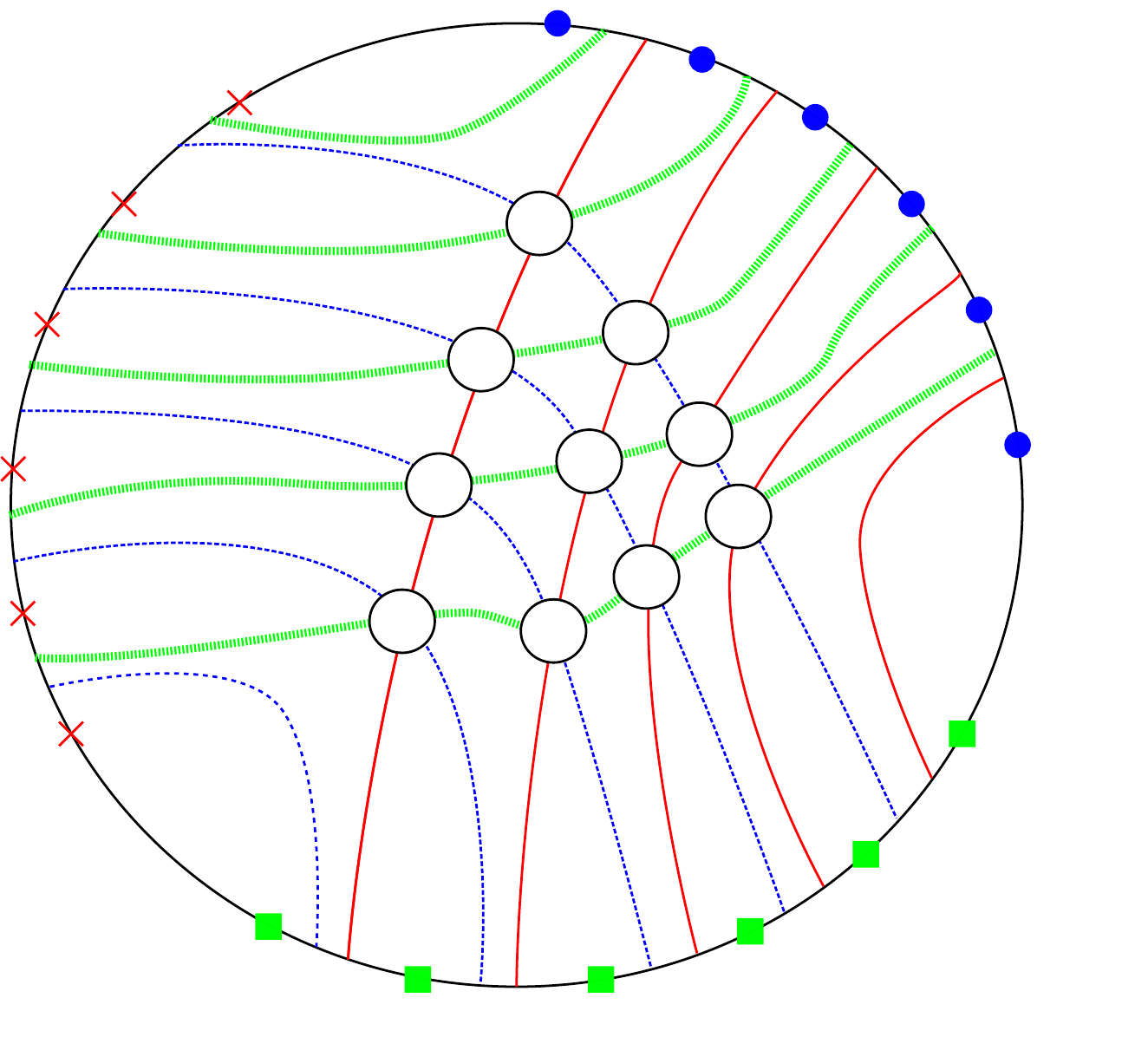}
\\ \\ a) & b) & c)
\end{tabular}
\caption{}
\label{fig:g>0}
\end{figure}
\end{proof}

Even if this will eventually 
follows from Theorem~\ref{thm:pseudoharnack}, 
we do not claim that the disk gluing in the proof of
Lemma~\ref{lem:g>0}  has any interpretation in terms of 
degenerations of $\phi(C)$.
Note that when $g=\frac{(d-1)(d-2)}{2}$, the arrangement 
$\bigcup_{i=0}^2\Gamma_i$  only depends,  up to 
orientation preserving homeomorphism, on $d$ and $s$. See
Figure~\ref{fig:g>0}c in the case $d=6$.

\subsection{Application of Rokhlin's complex orientation formula}

To end the proof of Theorem~\ref{thm:pseudoharnack} in the case
$d=2k$, it remains   to 
prove the following lemma. 
The case of curves of odd degree is entirely
similar, and is left to the reader.
\begin{lemma}
\begin{enumerate}
The following hold:
\item  $\phi(\gamma)$ bounds a disk in $\R P^2$ disjoint from
  $\R \phi(C\setminus\gamma)$ for any connected component $\gamma$ of $\R
  C\setminus \mathcal O$;
\item  a connected component of $\R\phi(C\setminus\mathcal O)$ is
contained in the disk bounded by $\phi(\mathcal O)$ in $\R P^2$  if and only if it
is contained in the quadrant $(+,+)$.
\end{enumerate}
\end{lemma}
\begin{proof}
These two facts will be a consequence of Rokhlin's complex orientation
formula (\cite{Rok} see also \cite{V3}).
Since there exists a smoothing  $\phi'(C')$ of $\phi(C)$ where $\phi':C\to\C
P^2$ is a real $J'$-holomorphic curve of degree $d$ and genus
$\frac{(d-1)(d-2)}{2}$, 
 we may assume\footnote{This assumption is aimed to simplify the exposition,
  and is not formally needed  for our 
  purposes. Indeed, 
there exists a generalization of Rokhlin's
  formula  for
  nodal curves that we could also have used here (\cite{Zvon83} see also
\cite{Vir96}).} from now on that $C$ has genus
$\frac{(d-1)(d-2)}{2}$. Analogously, we
 may further assume for simplicity that $\phi(C)$ 
intersects transversely  the three $J$-holomorphic
lines $L_i$.

Recall that since $C$ is maximal, the set $C\setminus \R C$ has two
connected components. Moreover the choice of one of these components
induces an orientation of $\R C$ (as  boundary). The effect of choosing the 
other component of $C\setminus \R C$ is to reverse the orientation 
of $\R C$. 
Hence there is a canonical orientation, up to a global change of orientation
of $\R C$, of all connected components of $\R C$. This orientation is called the
\emph{complex orientation of $\R C$}. 

Recall also that a disjoint pair of embedded circles in $\R P^2$ is
said to be \emph{injective} if their union bounds an annulus $A$. If
the two circles are oriented and form an injective pair, this latter
is said to be \emph{positive} if the two
orientations is induced by some orientation of  $A$, 
 and is said to be  \emph{negative} otherwise,
see
Figure~\ref{fig:orientation}a and b.
\begin{figure}[h!]
\centering
\begin{tabular}{ccccc} 
\includegraphics[height=5cm, angle=0]{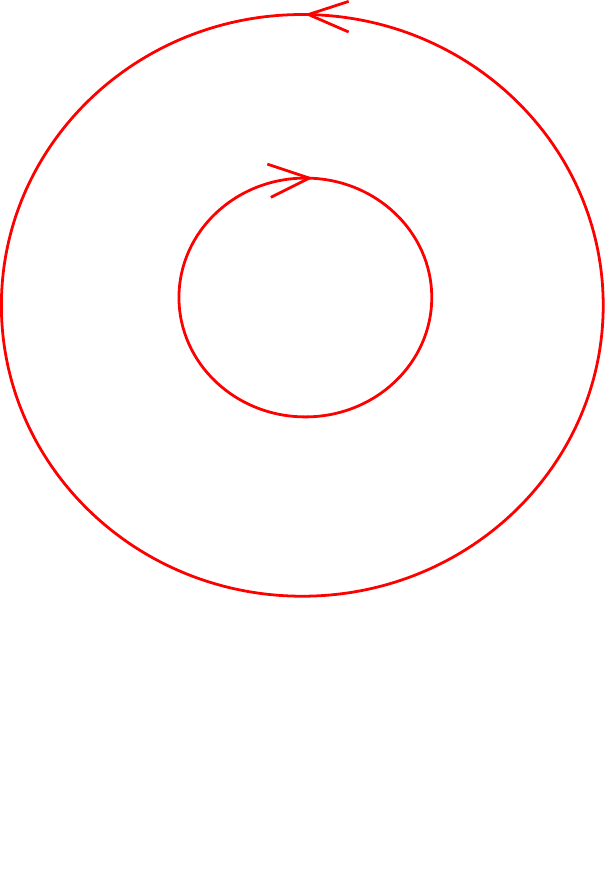} &
\hspace{3ex} &
\includegraphics[height=5cm, angle=0]{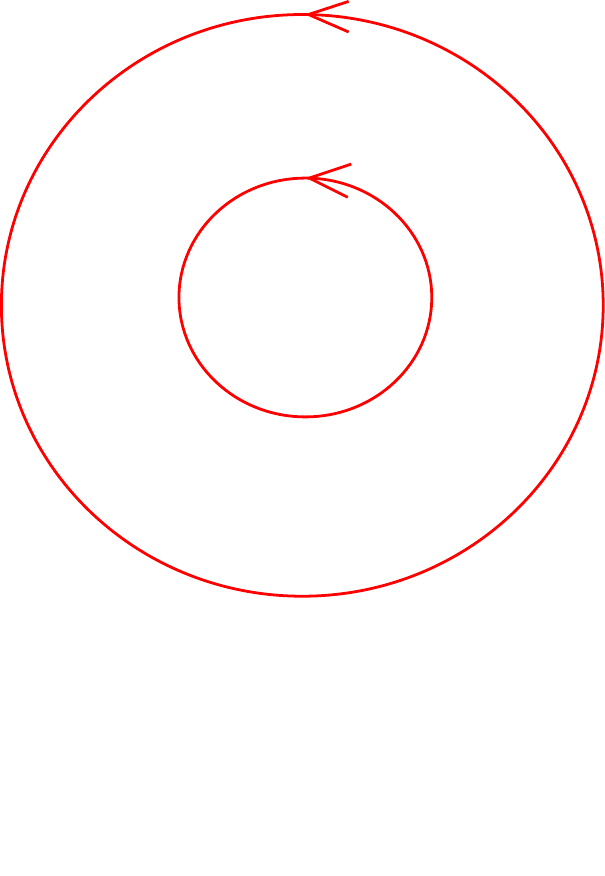} & 
\hspace{3ex} &
\includegraphics[height=7cm, angle=0]{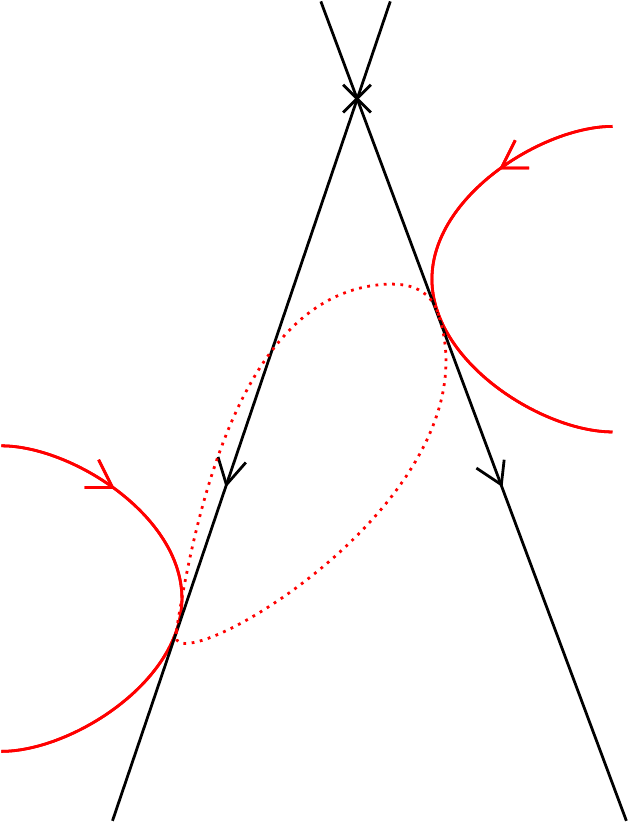}
\put(-105, 135){$\R D_1$}
\put(-90, 175){$p_{1,2}$}
\put(-0, 15){$\R D_2$}
\put(-140, 45){$\phi(q_1)$}
\put(-40, 125){$\phi(q_2)$}
\put(-70, 50){$\phi(\overline a)$}
\\ \\ a) A positive pair && b) A negative pair && c)  Fiedler's orientation rule
\end{tabular}
\caption{}
\label{fig:orientation}
\end{figure}
We denote respectively by $\Pi_+$ and $\Pi_-$ the number of positive
and negative injective pairs of connected components of $\phi(\R C)$
equipped with their complex orientation.
Rokhlin's complex orientation
formula reduces in our case to
\begin{equation}\label{equ:rokhlin}
\Pi_+-\Pi_-=\frac{(k-1)(k-2)}{2}.
\end{equation}

Now we  apply Fiedler's orientation rule (\cite{Fie2} see also
\cite{V3}) to estimate the quantities $\Pi_+$ and $\Pi_-$.
Consider the projection $\pi_0:C\to L_0$, and choose an arc $a$ of
$\Gamma_0\setminus \R C$. The arc $a$ lifts to a pair
of $conj_C$-conjugated arcs in $C$, whose topological closure in $C$,
 denoted by $\overline a$, is
homeomorphic to $S^1$. The set $\overline a\cap \R C$ consists of two
ramification points $q_1$ and $q_2$ of $\pi_0$. 
By construction, each of these two points $q_i$
corresponds to a tangency of $\phi(C)$ with a real $J$-holomorphic
line $D_i$
passing through $p_{1,2}$. 
Choose a complex orientation of $\R C$, and orient $\R D_1$ in a
compatible way with the complex orientation of $\R \phi(C)$ at $\phi(q_1)$, see
Figure~\ref{fig:orientation}c. Transport this orientation to $\R D_2$
via the portion of the pencil of $J$-holomorphic lines through
$p_{1,2}$ that intersect $\phi(\overline a)$. Fiedler's orientation
rule states that this orientation of $\R D_2$ is still compatible 
with the complex orientation of $\R \phi(C)$ at $\phi(q_2)$, see
Figure~\ref{fig:orientation}c.
\begin{figure}[h!]
\centering
\begin{tabular}{c} 
\includegraphics[height=7cm, angle=0]{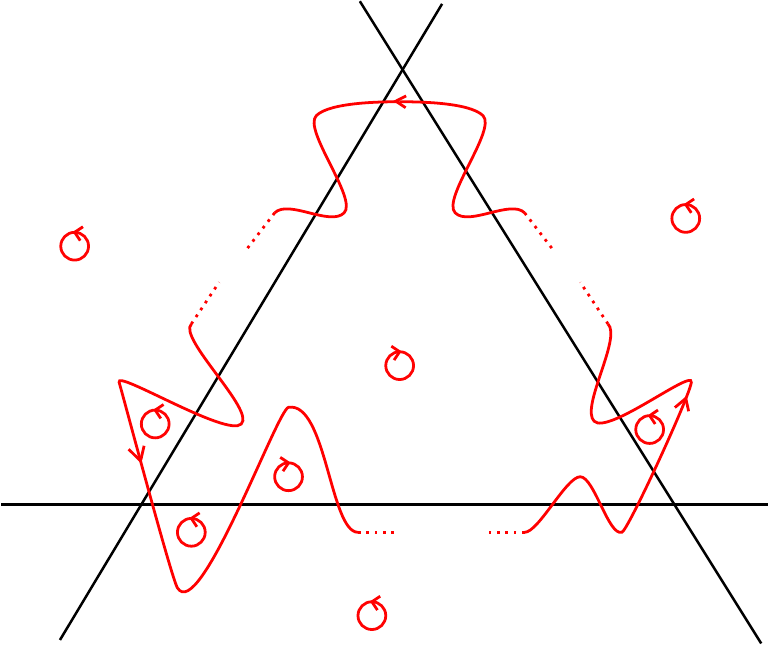}
\end{tabular}
\caption{}
\label{fig:orientation2}
\end{figure}

It follows from Lemmas~\ref{lem:g=0} and~\ref{lem:g>0}
that $\phi(q_1)$ is contained in the quadrant
$(\varepsilon_1,\varepsilon_2)$ if and only if 
$\phi(q_2)$ is contained in the quadrant
$(\varepsilon_1,-\varepsilon_2)$, see Figures~\ref{fig:disk2} and
\ref{fig:g>0}.
Hence Fiedler's orientation
rule implies that the complex
orientation of the curve $\phi(C)$ is as depicted in
Figure~\ref{fig:orientation2}.
In particular if $\gamma_1$ and $\gamma_2$ are two distinct connected
components of $\phi(\R C)$ which form an injective pair, we see that
this pair contributes to $\Pi_+$ if and only if
$\gamma_i=\phi(\mathcal O)$ and $\gamma_{3-i}$ is in the quadrant
$(+,+)$. Hence we deduce from Lemma~\ref{lem:g>0} 
that
$$\Pi_+\le   \frac{(k-1)(k-2)}{2}\quad\mbox{and}\quad 
\Pi_-\ge 0,$$
with equality if and only if the conclusion of the lemma holds.
Now the result follows from
Equation~$(\ref{equ:rokhlin})$.
\end{proof}

\begin{rem}
It is proved in \cite{Mik11} that the index map defined in \cite{Pass1}
provides a pairing between  
 connected components of 
$\R \phi(C\setminus \mathcal O)$ and  points with integer
coordinates in the interior of the triangle $\Delta_d$ with vertices
$(0,0)$, $(d,0)$, and $(0,d)$. 
It is interesting  that this pairing is also visible from the arrangements
$\Gamma_0\cup\Gamma_1\cup\Gamma_2$, see Figures~\ref{fig:disk2} and
\ref{fig:g>0}. In addition to the pairing, a triangulation of
$\Delta_d$ (dual to a honeycomb tropical
curve)
is also visible
in these pictures. I do not know whether this subdivision has any
interpretation.
 \end{rem}

\appendix
\section{}
As   consequences of Theorem~\ref{thm:harnack},  
we generalize to  simple $J$-holomorphic Harnack
curves some facts that are well known for 
simple algebraic Harnack
curves. 

\subsection{Topological types of simple Harnack curves}\label{app:all harnack}
Here we deduce from Theorem~\ref{thm:pseudoharnack}
all topological types of pairs 
$\left(\R
P^2, \R \phi(C)\bigcup \cup_{i=0}^2\R L_i\right)$, where $\phi:C\to \C
P^2$ is
 a simple Harnack
curve.

\begin{prop}
Let $\phi:C\to \C
P^2$ be
 a simple $J$-holomorphic Harnack
curve of degree $d$. 
Then the topological type of the pair 
$\left(\R
P^2, \R \phi(C)\bigcup \cup_{i=0}^2\R L_i\right)$ is obtained from
Figure~\ref{fig:harnack} by performing finitely many of the two
 following operations :
\begin{itemize}
\item the contraction of a circle disjoint from $\cup_{i=0}^2\R L_i$
  to a point, see Figure~\ref{fig:degeneration}a;
\item the replacement of $u_j$ consecutive intersection points with
  $\R L_i$ by a  point with of order of contact $u_j$, see Figure~\ref{fig:degeneration}b.
\begin{figure}[h!]
\centering
\begin{tabular}{ccc} 
\includegraphics[height=2.5cm, angle=0]{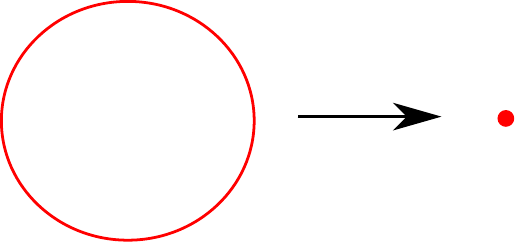}&
\hspace{18ex} &
\includegraphics[height=2.5cm, angle=0]{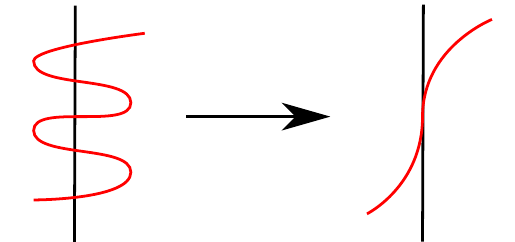}
\\ \\ a) && b)
\end{tabular}
\caption{}
\label{fig:degeneration}
\end{figure}

\end{itemize}

Conversely, any such topological type is realized by a $J$-holomorphic Harnack
curve of degree $d$.
\end{prop}
\begin{proof}
Indeed, let $\phi':C'\to \C P^2$ be 
 a simple $J'$-holomorphic  Harnack
curve of degree $d$ and genus $\frac{(d-1)(d-2)}{2}$ 
such that $\phi'(C')$ is a smoothing of $\phi(C)$,
and $\phi'(C')$ intersects transversely a $J'$-holomorphic
perturbation $L'_i$ of  $L_i$ for $i=0,1,2$. According the proof
of Theorem~\ref{thm:pseudoharnack}, the topological type of the pair 
$\left(\R
P^2, \R \phi'(C')\bigcup \cup_{i=0}^2\R L'_i\right)$ is given
Figure~\ref{fig:harnack}. This proves that the topological type of the pair 
$\left(\R
P^2, \R \phi(C)\bigcup \cup_{i=0}^2\R L_i\right)$ is as stated in the
proposition.

Analogously, to prove the second statement,
it is enough to exhibit a rational  Harnack
curve of degree $d$ intersecting each lines $L_i$ in a single point
 of order of contact $d$. According to Theorem~\ref{thm:pseudoharnack}, the map
$$\begin{array}{cccc}
\phi: & \C P^1 &\longrightarrow & \C P^2
\\ & [x:y]&\longmapsto & [x^d:y^d: (x-y)^d]
\end{array} $$
is  such a rational  Harnack
curve.
\end{proof}

\subsection{Simple Harnack curves in other toric
  surfaces}\label{app:other toric}
Here we deduce the classification of simple Harnack curves in any
toric surface from the classification of simple Harnack curves in $\C
P^2$.
Theorem~\ref{thm:pseudoharnack2} below 
can be proved along the same lines as Theorem~\ref{thm:pseudoharnack}. The reason why we restricted to 
 $\C P^2$ in Theorem~\ref{thm:pseudoharnack} is that,  
 thanks to symmetries, the proof 
in this particular case is
much more transparent and
avoids purely technical complications.
Furthermore  
Theorem~\ref{thm:pseudoharnack2}
 can be deduced from Theorem~\ref{thm:pseudoharnack} thanks to
Viro's patchworking. We briefly indicate below how to perform this reduction. We refer to
 \cite{V1,V4,Sh3} for references to patchworking, 
and \cite{IS} for its $J$-holomorphic version.


\medskip
Let $\Delta\subset \R^2$ be a convex polygon with vertices in $\Z^2$,
and let $X_\Delta$ be the complex algebraic toric surface associated to
$\Delta$, see \cite{GKZ}. 
The complement of the maximal toric orbit of $X_\Delta$ is denoted by $\partial
X_\Delta$, and is called the \emph{toric boundary} of $X_\Delta$.
There is a natural correspondence $e\leftrightarrow X_e$ between edges 
of $\Delta$ and irreducible
components  of $\partial X_\Delta$, which satisfies $e\cap
e'\ne\emptyset$ if and only if $ X_e\cap X_{e'}\ne\emptyset$.
Note that $X_\Delta$ might have isolated
singularities located at intersections $ X_e\cap X_{e'}$ 
of irreducible components of
$\partial X_\Delta$. Recall that $\Delta$ induces an embedding of
$X_\Delta$ into some projective space $\C P^N$, and we equip $X_\Delta$
with the restriction, still denoted by $\omega_{FS}$,
 of the corresponding Fubini-Study symplectic
form.
An almost complex structure $J$ on $X_\Delta$ 
tamed by $\omega_{FS}$ is said to be
\emph{compatible} if it coincides with the toric complex structure on
$X_\Delta$ in a neighborhood of $\partial X_\Delta$, and \emph{real}
if the standard complex conjugation on $(\C^*)^2=X_\Delta\setminus
\partial X_\Delta$ is $J$-antiholomorphic.

Let $(C,\omega)$ be a compact symplectic surface equipped with a 
complex structure $J_C$ tamed by $\omega$, and  a
$J_C$-antiholomorphic involution $conj_C$, and
let
 $J$ be a real  compatible almost complex structure on $X_\Delta$.
A real  $J$-holomorphic map 
 $\phi:C\to X_\Delta$ is said to have degree $\Delta$ if 
$\phi_*([C])$ is equal, in $H_2(X_\Delta;\Z)$, to the class realized by a
hyperplane section of $X_\Delta$ for the embedding induced by $\Delta$.
By the adjunction formula, a  $J$-holomorphic map 
 $\phi:C\to X_\Delta$ of degree $\Delta$ which does not factorize
through a non-trivial ramified covering has genus at most the number
of integer points in the interior $\overset{\circ}{\Delta}$ of $\Delta$.
Furthermore $ \phi(C)$ is non-singular in case of equality.

\begin{defi}\label{def:delta}
Let $\Delta\subset \R^2$ be a convex polygon with vertices in $\Z^2$,
and let $[e_1,\ldots,e_k]$ be the natural cyclic ordering on the edges of
$\Delta$.
A \emph{simple Harnack curve} of degree $\Delta$ 
is  a real  $J$-holomorphic map 
$\phi:C\to X_\Delta$ of degree $\Delta$, for some real  compatible almost complex structure
$J$ on $X_\Delta$,
satisfying the three following conditions:
\begin{itemize}
\item $C$ is a non-singular maximal real  curve;
\item there exist a connected component 
 $\mathcal O$ of $\R C$, and $k$ disjoint
  arcs $l_1,\ldots,l_k$
 contained in $\mathcal O$ such that $\phi^{-1}(X_{e_i})\subset l_i$;
\item the cyclic orientation  on the arcs $l_i$
induced by 
$\mathcal  O$ is precisely $[l_1,\ldots,l_k]$.
\end{itemize}
\end{defi}
Note that the last condition is non-empty only when $k\ge 4$.

\begin{exa}
For $\Delta_d$ the triangle with vertices $(0,0)$, $(d,0)$,
and $(0,d)$, the surface $X_{\Delta_d}$ is the projective plane
equipped with a homogeneous coordinate system, and $\partial X_{\Delta_d}$ is the
union of the three coordinate lines. A simple Harnack curve 
of degree $\Delta_d$ is a simple Harnack curve of degree $d$ in the
sense of Section~\ref{sec:intro}. Note however that a $J$-holomorphic 
simple Harnack
curve of degree $d$ might not be a simple Harnack curve
of degree $\Delta_d$, since $J$ is not required to be
integrable in a neighborhood of the coordinate axis. This additional
requirement is necessary when one wants to consider more general toric surfaces.
\end{exa}

As in Section~\ref{sec:intro}, given  $\phi:C\to X_\Delta$  a simple Harnack
curve,
we encode in a sequence the intersections of
$\phi(\mathcal O)$ with the components of $\partial X_\Delta$. 
The choice of  an orientation of $\mathcal O$ induces an ordering of 
the intersection points  of $\mathcal O$  with  $X_{e_i}$, and 
we
denote by $s_i$ the corresponding sequence of
intersection multiplicities. Let $s$ be the sequence $(s_1,\ldots ,s_k)$
considered up to the equivalence relation generated by
$$
(s_1,\ldots, s_k)\sim (\overline s_1,\ldots, \overline s_k),
\quad(s_1,\ldots, s_k)\sim (s_k,s_1,\ldots,s_{k-1}),$$
and
$$(s_1,\ldots ,s_k)\sim 
(s_k, s_{k-1},\ldots, s_1).$$
Recall that $\overline{(u_i)_{1\le i\le n}}=(u_{n-i})_{1\le i\le n}$.


\begin{thm}\label{thm:pseudoharnack2}
Let $\Delta\subset \R^2$ be a convex polygon with vertices in $\Z^2$,
and let $\phi:C\to X_\Delta$ be a simple Harnack curve of degree $\Delta$. Then 
the curve $\phi(C)$ 
has  solitary nodes 
 as only singularities (if any).
Moreover if either $g(C)=0$ or $g(C)=|\Z^2\cap \overset{\circ}{\Delta}|$, then
the topological type of the
pair $\left((\R^*)^2, \R \phi(C)\cap (\R^*)^2\right)$ only
depends on $\Delta$, 
$g(C)$, and $s$.

\end{thm}
\begin{proof}
Let us assume for simplicity that $\phi(C)$ intersects $\partial
X_{\Delta}$ transversely,
and suppose for a moment that we have proved the following:

\medskip
{\bf Claim: }for any edge  $e$  of $\Delta$, 
the cyclic  orders on the finite set $\mathcal O\cap \R X_{e}$ induced
by  $\mathcal O$ and $\R X_{e}$ coincide.

\medskip

Assuming this claim,
 one constructs exactly as in the proof of \cite[Theorem 2(1)]{KhRiSh01} a simple Harnack
curve in $\C P^2$ by patchworking  $\phi(C)$
with finitely many  simple
algebraic Harnack curves constructed in \cite{IV2}. Now Theorem~\ref{thm:pseudoharnack2} follows from Theorem~\ref{thm:pseudoharnack}.

 \medskip
Hence we are left to prove the claim. Let $e$ be an edge of 
$\Delta$, and define $\overset{\circ}{X_{e}}$ to be $X_{e}$ from which we
remove  its
two intersection points with the other irreducible components of 
$\partial X_\Delta$.
Since the almost complex structure on $\Delta$ is integrable
in a neighborhood of $\partial X_\Delta$, there exists a
$J$-holomorphic compactification of $(\C^*)^2\cup \overset{\circ}{X_{e}}$ into 
$\C P^2=(\C^*)^2\cup L_0\cup L_1\cup L_2 $ where  $L_i$ is a
$J$-holomorphic line in $\C P^2$, and $L_0$ is a compactification of 
$\overset{\circ}{X_{e}}$.
The map $\phi$ induces a $J$-holomorphic map $\phi':C\to \C P^2$, and
exactly as in the beginning of Section~\ref{sec:algebraic}, one  proves that
the map $\pi_0:C\to L_0$ has no ramification point  on the connected
component  of
$\mathcal O\setminus \phi'^{-1}(L_1\cup L_2)$ containing  
$\phi'^{-1}(L_0)$. This is precisely saying that the cyclic  orders on the set
 $\mathcal O\cap \R X_e$ induced
by  $\mathcal O$ and $\R X_e$ coincide.
\end{proof}

\bibliographystyle {alpha}
\bibliography {../../Biblio.bib}

\end{document}